\documentclass[12pt, twosided]{amsart}

\usepackage{amssymb, hyperref}
\usepackage{enumerate}
\usepackage{amsthm}
\usepackage[mathscr]{eucal}
\usepackage{xcolor}
\usepackage{tikz}
\usetikzlibrary{matrix,arrows}

\usepackage{fancyhdr}

\newtheorem{thm}{Theorem}[section]
\newtheorem{prop}[thm]{Proposition}
\newtheorem{cor}[thm]{Corollary}

\newtheorem{lemma}[thm]{Lemma}

\newtheorem*{introthmA}{Theorem A}
\newtheorem*{introthmB}{Theorem B}

\newtheorem*{introthmE}{Theorem E}
\newtheorem*{introthmF}{Theorem F}

\theoremstyle{definition}
\newtheorem{defn}[thm]{Definition}

\theoremstyle{remark}
\newtheorem{remark}[thm]{Remark}
\newtheorem{e.g.}[thm]{Example}
\newtheorem{notrem}[thm]{Notational remark}

\newcommand{\lvC}{{\mathbb C}}

\newcommand{\lvP}{{\mathbb P}}

\newcommand{\lvR}{{\mathbb R}}
\newcommand{\lvZ}{{\mathbb Z}}

\newcommand{\sA}{\mathscr{A}}
\newcommand{\sC}{\mathscr{C}}
\newcommand{\sE}{\mathscr{E}}
\newcommand{\sF}{\mathscr{F}}

\newcommand{\sH}{\mathscr{H}}

\newcommand{\sP}{\mathscr{P}}
\newcommand{\sO}{\mathscr{O}}

\newcommand{\sU}{\mathscr{U}}

\newcommand{\vect}{\mathrm{Vect}}
\newcommand{\Rep}{{\mathrm{Rep}}}
\renewcommand{\Re}{{\mathfrak{R}}\mathfrak{e}\;}
\newcommand{\un}{\rm un}

\newcommand{\iso}{\stackrel{\sim}{\rightarrow}}

\newcommand{\ora}{\overrightarrow}
\newcommand{\ol}{\overline}
\newcommand{\Pomt}{\lvP^1 \setminus \{0,1,\infty\}}
\newcommand{\Aut}{\mathrm{Aut}}
\newcommand{\Res}{\mathrm{Res}}

\newcommand{\Li}{\mathrm{Li}}
\newcommand{\uni}{\mathrm{un}}

\newcommand{\ups}{\upsilon}
\newcommand{\Cs}{{\lvC\ll\! \!X_0,X_1\!\! \gg}}

\begin{document}

\title{On an extension of the universal monodromy representation for  $\mathbb{P}^1\backslash\{0,1,\infty\}$}
\author{$\;\;$Sheldon T Joyner
\\
\tiny{Brandeis University}
}
\maketitle

\begin{abstract} The {\sc Chen} series map giving the universal monodromy representation of $\Pomt$ is extended to an injective 1-cocycle of $PSL(2, \lvZ)$ into power series with complex coefficients in two non-commuting variables, twisted by an action of $S_3.$ The definition of the 1-cocycle is effected by parallel  transport of flat sections of the bundle, also with an $S_3$ twisting, along paths in $\Pomt$ which are explicitly associated to elements of $PSL(2, \lvZ)$. 
The resulting action of $PSL(2, \lvZ)$ on the polylogarithm generating function is shown to yield a family of proofs of the analytic continuation and functional equation of the {\sc Riemann} zeta function.
\end{abstract}

\tableofcontents

\section*{Introduction}
As is well-known, the monodromy representation corresponding to the universal prounipotent bundle $\sU$ with connection $\nabla$ on $\Pomt$ (cf. \S 1 of \cite{Kim:UAM}) may be described by means of the Chen series map on homotopy classes of paths $[\gamma] \in \pi_1(\Pomt, c)$:
\begin{equation}
\label{e:Chen}
[\gamma] \mapsto
\sum_{w}\int_{\gamma}\omega_{i_1} \ldots \omega_{i_k}X_{i_1} \ldots X_{i_k}
\end{equation}
where the sum is taken over all words in the non-commuting formal variables $X_0$ and $X_1$ (including the empty word, for which the corresponding integral is 1), $c$ is any (possibly tangential) basepoint,  and if $z$ denotes the usual parameter on $\Pomt,$ 
$\omega_{0} = \frac{dz}{z}$ while 
$\omega_{1} = \frac{dz}{1-z}$
(see Proposition 11 in \cite{Hain:Lec}). When $c$ is the tangential basepoint $\ora{01}$ and $\gamma$ is a path from $c$ to $z$ which does not cut the real axis unless $z$ is real, the series which results 
is called the polylogarithm generating series, and is denoted $\Li(z, X_0,X_1)$ or $\Li(z)$ for short. Here, the integrals which appear are regularized in the usual way - cf. \cite{Joyner:ii}. For such $\gamma$, the coefficients of the terms of the form of $X_1X_0^{n_1}X_1X_0^{n_1}\cdots X_1 X_0^{n_r}$ are the multiple polylogarithm functions.

Because the bundle $\sU$ is given by
\[
\sU = \lim_{\leftarrow_N} \left[
\frac{\lvC\!<\!X_0,X_1\!>}{(X_0,X_1)^{N+1}} \otimes \sO_{\Pomt} \right]
\]
and $\nabla$ is the formal {\sc Knizhnik-Zamolodchikov} connection 
\[
\nabla = d -\left(\frac{dz}{z}X_0 + \frac{dz}{1-z}X_1\right),
\]
one verifies without difficulty that $\Li(z)$ is a flat section of $(\sU, \nabla)$.

Here we prove:
\begin{introthmA}
The monodromy representation
\[
F_{\bullet}:\pi_1(\Pomt, \ora{01}) \rightarrow \lvC\!\ll \! X_0,X_1\!\gg^{\times}
\]
admits an extension to an injective 1-cocycle
\[
F_{\bullet}: PSL(2, \lvZ) \rightarrow \lvC\!\ll \!X_0,X_1\!\gg_{\Lambda}^{\times}
.
\]
\end{introthmA} 
(See \ref{t:PSL.act} and \ref{t:cocycle} below.) 

 $\lvC\!\ll \! X_0,X_1\!\gg$ denotes the algebra of power series with complex coefficients in 
$X_0$ and $X_1$, and 
$\lvC\!\ll \! X_0,X_1\!\gg_{\Lambda}^{\times}$ denotes invertible power series with an action of $PSL(2, \lvZ)$ which factors through $\Lambda \simeq S_3$ via the usual surjection $\sA$ of (\ref{e:ses}) below. The $S_3$ action on power series is induced by the action of the group of automorphisms $\Lambda$ of $\Pomt$ on the connection $\nabla$, and was given in (25) of \cite{OU}. It is also described in \S \ref{sss:top} below. 

The existence of this extension is facilitated by the following short exact sequence:
\begin{equation}
\label{e:ses}
1\rightarrow \Gamma(2)/\{\pm 1\} \simeq \pi_1(\Pomt,c)
\rightarrow PSL(2, \lvZ) \stackrel{\sA}{\rightarrow} 
SL(2, \lvZ/2\lvZ)
\rightarrow 1.
\end{equation}
First, we give an explicit lifting of 
the fundamental group $PSL(2, \lvZ)$ of the orbifold $\left[ PSL(2, \lvZ)\backslash H\right]$ (where $H$ is the upper half plane) to a space of homotopy classes of paths in $\Pomt$. Then, parallel transport along these paths, twisted by the above-mentioned action of $S_3$, gives an action of $PSL(2, \lvZ)$ on sections of $(\sU, \nabla)$. Where the flat section $\Li(z)$ is concerned, we prove in Proposition \ref{t:PSL.act} below that this action amounts to multiplying the section by a power series. This power series gives the extension of the monodromy representation to a 1-cocycle on $PSL(2, \lvZ)$, and is given by the formula: 
\[
\alpha \;\;\mapsto \;\;\left.
\sum_{w}\int_{\alpha}\omega_{i_1} \ldots \omega_{i_k}Y_{i_1}\ldots Y_{i_k}\right|_{Y_{i_j} = \ol\alpha X_{i_j},}
\]
with sum and $\omega_{i_j}$ notation as above, writing $\ol\alpha$ for the reduction of $\alpha \in PSL(2, \lvZ)$ to $SL(2, \lvZ/2\lvZ)\simeq S_3,$ and $\ol\alpha X_{i_j}$ for the action of $\ol\alpha \in S_3$ on $X_{i_j}.$ 

The proof of the injectivity rests on {\sc Chen's} $\pi_1$ {\sc DeRham} Theorem, (cf. Theorem 10 of \cite{Hain:Lec}).

Extending the monodromy representation to $PSL(2, \lvZ)$ yields an additional symmetry on $\Li(z)$ which can be used to prove the analytic continuation of {\sc Riemann's} zeta function $\zeta(s)$.  
This allows us to draw parallels between 
the classical theta function technique used to prove the  analytic continuation and functional equation of $\zeta(s)$, and {\sc Riemann's} original contour integral approach. As {\sc Hecke} noticed in \cite{Hecke}, the following two facts comprise the essence of the theta function proof:
 \\
 {\bf  T0.} The {\sc Jacobi} theta function $\theta(\tau, z)$ is modular in $\tau$ in the usual sense, with respect to the congruence subgroup $\Gamma(2)$ of $PSL(2, \lvZ)$. (This explains the existence of the {\sc Fourier} series expansion for $\theta(\tau,z)$.)
 \\
 {\bf T1.} $\theta(\tau,z)$ satisfies an additional symmetry property with respect to the involutive generator $\sigma$ of $PSL(2, \lvZ)$ (given by ${\sigma}: \tau \mapsto -1/\tau$ in the action on $H$), namely the functional equation of $\theta(\tau,z)$, which is regarded as an additional modularity property in $\tau$.
  
Here, we show that {\sc Riemann's} contour integral expression for $\zeta(s)$ fits into the context of a family of integral expressions, each of which may be used to prove the analytic continuation and functional equation for $\zeta(s).$ Taken together, these proofs result from the following facts:
\\
{\bf P0.} The monodromy of the polylogarithm generating function $\Li(z)$ may be calculated (as for example in \cite{MPvdH}) by directly performing the analytic continuation along the paths of the fundamental group of $\Pomt$. The equations which result may be thought of as transformation rules for $\Li(z)$ with respect to elements of  $\Gamma(2)/\{\pm1\} \simeq \pi_1( \Pomt, c)$ (where $c$ is any basepoint - possibly tangential). 
\\
{\bf P1.} $\Li(z)$ satisfies an additional symmetry property with respect to ${\sigma} \in PSL(2, \lvZ)$, namely a functional equation involving the {\sc Drinfel'd} associator. Coinciding as it does with the action on $\Li(z)$ of an element of $S_3$ (cf. \S 3 of \cite{OU}), this symmetry property is well-known,
but here it is 
shown to arise from extending a universal monodromy representation of the fundamental group of $\Pomt$, to a 1-cocycle on $PSL(2, \lvZ)$.

Property P1 facilitates the analytic continuation in that it gives rise to the {\sc Euler} connection formulae (see Proposition 5 in \cite{OU}), which allow us to avoid non-integrable monodromy terms by shifting monodromy of the integrands from $0\in \lvC$ to $\infty$ - for the details see \S \ref{s:AC}.

\section*{Acknowledgements}
The author is glad of the chance to express his thanks to 
{\sc Minhyong Kim} for his unwavering encouragement and patient explanations, as well as for sharing his insights.

\section{The extension of the monodromy representation}
\subsection{Explicit lifting of $PSL(2, \lvZ)$ to classes of paths in $\Pomt$}
\label{sss:top}

Suppose that $X = \overline{X} \setminus S$ is a smooth curve over $\lvC$ where $S$ is some finite set of points. In \cite{Deligne:89}, {\sc Deligne} introduced a notion of fundamental group of $X$ based at any given omitted point $a \in S$, in the direction of some specified tangent vector to $\overline{X}$ at $a$.
Classically, as in \cite{Hain:CP} such fundamental groups with {\em tangential basepoint} may be defined as follows:  If $\vec{v_j} \in T_{a_j}$ is a tangent vector at $a_j \in S$ for $j=0,1$,  set
\[
P_{\vec{v_0},\vec{v_1}}:=\{ \gamma :[0,1] \rightarrow \overline{X}| \gamma'(0) = \vec{v_0}, \gamma'(1) = -\vec{v_1}, \gamma((0,1)) \subset X\}.
\]
\begin{defn}
The fundamental path space 
$\pi_1(X, \vec{v_0}, \vec{v_1})$ is the set of path components of $P_{\vec{v_0},\vec{v_1}}.$
\newline When $\vec{v_1}=\vec{v_0},$ this is the fundamental group denoted $\pi_1(X,\vec{v_0}).$
\end{defn}

This naive description is sufficient for the use of the paper. For our purposes, $\overline{X} = \lvP^{1}_{\lvC}$, $S=\{0,1,\infty\}$, and 
$\ora{ab}$ will denote the tangent vector of unit length over $\overline{X}$ at $a \in S,$ pointing in the direction of $b \in S$ for any $b \neq a$.

\begin{defn}
Any fundamental path space of the form of 
\[
\pi_{1}(\Pomt, \ora{a_0b_0}, \ora{a_1b_1})
\]
where $a_j, b_{j} \in \{0, 1,\infty\}$ and $a_{j} \neq b_{j}$ for $j=0,1$ will be called a {\em real-based fundamental path space} of $\Pomt,$ and the tangential basepoints $\ora{a_jb_j}$ will be referred to as real tangential basepoints.
\end{defn}

Fix a real tangential basepoint $\ora{ab}.$ Then form the set 
\[
G_{\ora{ab}}:=\cup \pi_1(\Pomt, \ora{ab}, \ora{a_0b_0})
\]
where the union is taken over all $a_0,b_0 \in \{0,1,\infty\}$ with $a_0 \neq b_0.$ The utility of restricting attention to the real tangential basepoints lies in the fact that they admit an action of $SL(2,\lvZ/2\lvZ)$ (see below). Using this action, $G_{\ora{ab}}$ will be endowed with a group structure, by means of which it can be identified with $PSL(2,\lvZ).$

Now as is described in \cite{Chandra},  the symmetries of the classical $\lambda$ function effecting the covering of $\Pomt$ by $H$ are captured by the classical anharmonic group $\Lambda$, to which $SL(2, \lvZ/2\lvZ)$ is isomorphic. $\Lambda$ is given explicitly as the following group of linear fractional automorphisms of $\Pomt$:
\[
\Lambda = \left\{\lambda \mapsto\lambda, \lambda \mapsto1-\lambda, \lambda \mapsto\frac{\lambda}{\lambda-1}, \lambda \mapsto\frac{1}{\lambda},\lambda \mapsto \frac{\lambda-1}{\lambda}, \lambda \mapsto\frac{1}{1-\lambda}
\right\}.
\] 
It is evident from the topology that $\Lambda$ is exactly the group of {\em all} such linear fractional automorphisms of $\Pomt$. 

Note that these transformations necessarily permute the real tangential basepoints, as is also immediate from the above explicit description. In fact, the elements of $\Lambda$ are characterized by the corresponding permutations of the symbols $0$, $1$ and $\infty$ so that also $\Lambda  \simeq S_3$.

Once and for all fix isomorphisms
\begin{equation}
\label{e:red.gp}
SL(2, \lvZ/2\lvZ) \simeq \Lambda \simeq S_3
\simeq
< \ol{\sigma}, \ol{\rho} | \ol{\sigma}^2 = \ol\rho ^2=1;
\ol\sigma \ol\rho \ol\sigma
=
\ol\rho \ol\sigma \ol\rho>
\end{equation}
by identifying the respective generators
\[
\left[ \begin{array}{cc}
0&1\\
1&0\end{array}\right]
\leftrightarrow
(\lambda \mapsto 1-\lambda)
\leftrightarrow
(01)
\leftrightarrow
\ol\sigma
\]
and
\[
\left[\begin{array}{cc}
 1&1\\
 0&1
 \end{array}
 \right]
 \leftrightarrow
 (\lambda \mapsto \frac{\lambda}{\lambda-1})
 \leftrightarrow
 (1\infty)
 \leftrightarrow
 \ol\rho.
 \]

Now suppose given the real tangential basepoint $\ora{ab} = \ora{01}.$ Then let $s$ denote the homotopy class of paths in $\Pomt$ represented by the tangential path $[0,1]$ and let $r$ be the homotopy class of paths represented by the loop from $\ora{01}$ to $\ora{0\infty}$ in the upper half plane, as pictured below.

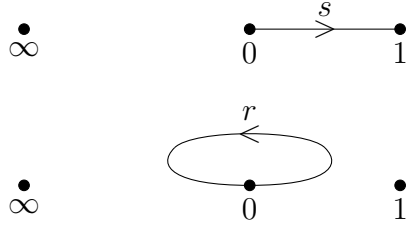
\begin{center}
\begin{tikzpicture}
\filldraw (0,0) node [below=1pt] {0} circle (2pt);
\filldraw (2,0) node [below=1pt] {1} circle (2pt);
\filldraw (-3,0) node[below=1pt] {$\infty$} circle (2pt);
\draw (0,0) --node [above = 1pt] {$s$} node {$>$}  (2,0);
\end{tikzpicture}
\end{center}

\begin{center}
\begin{tikzpicture}
\draw (0,0) .. controls (1, 0) and (1.25,.25) .. (1,.5);
\draw (1,.5) .. controls (.75,.75) and (-.75,.75) .. node {$<$} node [above = 1pt] {$r$} (-1,.5);
\draw (-1,.5) .. controls (-1.25,0.25) and (-1,0) .. (0,0);
\filldraw (0,0) node [below=1pt] {0} circle (2pt);
\filldraw (2,0) node [below=1pt] {1} circle (2pt);
\filldraw (-3,0) node[below=1pt] {$\infty$} circle (2pt);
\end{tikzpicture}
\end{center}

The use of tangential basepoints prevents homotopies which would otherwise occur - in particular, the homotopy classes can detect an {\em upper half plane} owing to the rigidity of the real line with respect to a choice of a pair of real tangential basepoints. In this way, one sees that $r$ is well-defined as a homotopy class of paths which differs from the class of a similar loop in the lower half plane.

The group structure on $G_{\ora{ab}}$ is facilitated by the distinct presentations of $SL(2, \lvZ/2\lvZ)$ coming from (\ref{e:red.gp}): Firstly we define the surjection $[ \; \cdot \;]_{ab}$ of $G_{\ora{ab}}$ onto $SL(2, \lvZ/2\lvZ) \simeq S_3$ by sending a given homotopy class $t$ in $G_{\ora{ab}}$ with endpoint $\ora{a_tb_t}$, to the permutation $[t]_{ab}$ of $\{0,1,\infty\}$ sending $a$ to $a_t$ and $b$ to $b_t$. 
Next, we exploit the fact that the fractional linear automorphisms $\Lambda$ are also isomorphic to $SL(2, \lvZ/2\lvZ)$ to define an action of this group on $G_{\ora{ab}}:$ Any $\ol\alpha\in \Lambda$ is a self-mapping of $\Pomt$ and as such sends any homotopy class $u$ of paths between real tangential basepoints, to some other such  homotopy class of paths. We denote the latter by $\ol\alpha * u$.

Sythesizing these definitions, we have a map of $G_{\ora{ab}}\times G_{\ora{ab}}$ into $G_{\ora{ab}}$ given by 
\[
(t,u) \mapsto [t]_{ab}*u.
\]


\begin{remark}
\label{r:rs}
When $\ora{ab} = \ora{01}$ we write $[\;\cdot\;]$ for $[\;\cdot\;]_{01}$. Then notice that, viewed as linear fractional transformations of $\Pomt$, 
\[
[r]:z \mapsto \frac{z}{z-1}
\] 
while 
\[
[s]:z\mapsto 1-z.
\] 

Furthermore, for any $t \in G_{\ora{01}}$ with endpoint $\ora{a_1b_1}$, one checks by direct computation that $[t]*r$ may be represented by a loop in the upper or lower half plane (according to the corresponding permutation $[t]$ being even or odd respectively), beginning at $\ora{a_1b_1}$ and ending at $\ora{a_1c_1}$ where $c_1 \neq b_1,$ while $[t]*s$ may be represented by a straight line segment beginning at $\ora{a_1b_1}$ and ending at $\ora{b_1a_1}$.
\end{remark}

Using this action, we define a concatenation procedure for homotopy classes of paths in $G_{\ora{01}}$ according to the following inductive prescription: 
If $\eta$ is a homotopy class of paths formed from the concatenation procedure applied successively to classes in $\{r,s\}$, and $\nu$ is either $r$ or $s$, let $\eta \nu$ be the homotopy class of $\eta$ followed by $[\eta]*\nu$.
Since $[\eta]$ sends $\nu$ to a homotopy class of paths originating at the endpoint of the paths in $\eta,$ it follows that $\eta\nu \in G_{\ora{01}}$.

The construction may be repeated for any choice of real tangential basepoint $\ora{ab}.$ In cases other than $\ora{ab} = \ora{01}$ write $r_{ab}$ and $s_{ab}$ for the corresponding generators. To be precise, $r_{ab}$ is a loop based at $\ora{ab}$ of the form of $r$ as above, which is in the upper half plane for $\ora{ab} = \ora{\infty 0}$ and $\ora{ab} = \ora{1\infty}$ but in the lower half plane when $\ora{ab}$ is $\ora{10},$ $\ora{0\infty}$, or $\ora{\infty 1};$ while $s_{ab}$ is a straight line segment from $\ora{ab}$ to $\ora{ba}.$ 

Throughout write $\cdot$ for concatenation of (homotopy classes of) paths.

\begin{defn}
The mapping 
\[
S_{ab}:
G_{\ora{ab}} \times G_{\ora{ab}} 
\rightarrow G_{\ora{ab}}
\]
with
\[
S_{ab}(\eta, \mu)= \eta\mu := \eta \cdot ([\eta]_{ab}*\mu)
\]
for any $\eta, \mu\in G_{\ora{ab}}$, 
will be referred to as $SL(2,\lvZ/2\lvZ)$ concatenation of tangential paths in $G_{\ora{ab}}.$
\end{defn}

One checks that for any $\eta, \mu \in G_{\ora{ab}},$
\begin{equation}
\label{e:aux0}
[\eta \mu] = [\eta] \circ [\mu].
\end{equation}
Using this fact, one readily proves the associativity of successive application of $S_{ab}$: I.e., for any $\eta, \mu$ and $\nu$ in $G_{\ora{ab}},$
\[
S_{ab}(\eta ,S_{ab}(\mu, \nu)) = S_{ab}(S_{ab}(\eta,\mu), \nu).
\]

Because of the associativity, for any $n \geq 1,$ the $SL(2, \lvZ/2\lvZ)$ concatenation $\nu_1 \ldots\nu_n$ of elements $\nu_j \in \{r,s\}$ is uniquely determined. It is given by
\[
\nu_1 \cdot ([\nu_1]*\nu_2 )\cdot
([\nu_1 \cdot [\nu_1]*\nu_2] * \nu_3) 
\cdot \ldots \cdot
([\nu_1\cdot [\nu_1]*\nu_2\cdot \ldots \cdot  [\ldots [[\nu_1 \cdot [\nu_1]*\nu_2]*\nu_3]\ldots]*\nu_{n-1}]*\nu_n),
\]
where $\cdot$ again denotes concatenation of (homotopy classes of) paths. 
Applying (\ref{e:aux0}) iteratively,
one sees that for any $m \leq n$, $$[\nu_1 \ldots \nu_m] = [\nu_1] \circ \ldots \circ [\nu_m],
$$ so $\nu_1 \ldots \nu_n$ may be rewritten
\[
\nu_1 \cdot ([\nu_1]*\nu_2)
\cdot (([\nu_1] \circ [\nu_2])*\nu_3)
\cdots (
([\nu_1] \circ \ldots \circ [\nu_{n-1}])*\nu_n).
\]

Now it is possible to show that for any real tangential basepoint $\ora{ab},$ $G_{\ora{ab}}$ may be endowed with a group structure with multiplication given by $S_{ab}.$  To simplify the notation, consider the case of $\ora{ab} = \ora{01}.$ Begin by observing that the class $e$ of the trivial path acts as the identity. Also, $s$ is its own inverse,  since $[s]*s$ is the homotopy class of paths represented by the tangential path $[1,0],$ which is inverse to $[0,1]$. The inverse of $r$ is the homotopy class $q$ of paths represented by the loop from $\ora{01}$ to $\ora{0\infty}$ in the {\em lower} half plane - one checks easily that $rq=qr=e.$ We write $q=r^{-1}.$ 
Of course $[r]= [r^{-1}].$ 

With the group structure induced in this way,  it is easy to prove that 
\[
G_{\ora{ab}} \simeq <\!r_{ab},s_{ab}\!>\!\!/(s_{ab}^2, (s_{ab}r_{ab})^3),
\]
where $<\!r_{ab},s_{ab}\!>=F_2$ denotes the free group on the two generators $r_{ab}$ and $s_{ab}$.

Now it is a well-known fact that
\[
PSL(2,\lvZ) = <\rho, \sigma>/(\sigma^2, (\rho \circ \sigma)^3)
\]
where $<\rho, \sigma> = F_2,$ 
the free group on two generators.
(For example, consult \cite{LangSL}, in which the {\sc Bruhat} decomposition is given, by means of which one can write down generators and relations for $SL(2,\lvR).$) Viewing $PSL(2, \lvZ)$ as a group of linear fractional transformations of $H$, generators may be given by
\[
\rho: \tau \mapsto  1+\tau
\]
and
\[
\sigma:\tau \mapsto -\frac{1}{\tau}.
\]

It then follows that for any real tangential basepoint $\ora{ab}$, 
\begin{equation}
\label{e:Iso}
G_{\ora{ab}}\simeq PSL(2, \lvZ). 
\end{equation}

Since we now have
\[
\pi_1(\Pomt, \ora{ab}) \;\;\lhd \;\;G_{\ora{ab}},
\]
the isomorphism of (\ref{e:Iso}) gives the isomorphism of the fundamental group with the congruence subgroup $\Gamma(2)/\{\pm1\}$ on the level of the generators $\rho$ and $\sigma$.

\begin{notrem}
The multiplication in $PSL(2, \lvZ)$ is written in the functional order, whereas concatenation of paths in $G_{\ora{ab}}$ occurs in the order in which the paths are written.
\end{notrem}

\begin{remark} 
\label{r:corr}
Denote the isomorphism 
of (\ref{e:Iso}) by 
\[
\Psi_{ab}: G_{\ora{ab}} \stackrel{\simeq}{\rightarrow}PSL(2, \lvZ),
\]
writing 
$\Psi:=\Psi_{01}$ in the special case of $\ora{ab}=\ora{01}.$ 
\end{remark}

We know that for any given $u \in G_{\ora{ab}},$ $\Psi_{ab}(u)$ is a transformation of the upper half plane which sends the lift of $\ora{ab}$ under the covering map $\lambda : H \rightarrow \Pomt$ in some fixed fundamental domain for $\Pomt$, to some lift of the endpoint of $u$ under $\lambda(\tau).$

Finally we remark that with notation as above, 
\[
[\Psi_{ab}^{-1}(\cdot)]_{ab}:PSL(2, \lvZ) \rightarrow SL(2, \lvZ/2\lvZ)
\]
is the usual projection  (i.e. $\sA$ of (\ref{e:ses})).

Subsequently write $\sA(\ups) = \ol{\ups}$ for any $\ups \in PSL(2, \lvZ),$ and suppress the mapping $\Psi_{ab}$ from the notation. (I.e. implicitly identify elements of $PSL(2, \lvZ)$ with those of $G_{\ora{ab}}$).

\subsection{Extending the monodromy representation}

\subsubsection{The universal prounipotent bundle with connection on $\Pomt$}
\label{sss:UPB}
For definitions and properties of {\sc Chen} iterated integrals, the reader is referred to \cite{Hain:Lec} or \cite{Joyner:ii}, and for general facts related to bundles with connections on curves (and parallel transport), to \cite{Deligne:RSP} or \cite{Joyner:thesis}. 

Concretely, the universal prounipotent bundle with connection (cf. \cite{Kim:UAM}) on $\Pomt$ is constructed as follows:
With $X_0$ and $X_1$ formal non-commuting variables as above and $I=(X_0,X_1)$ the augmentation ideal, let
\[
U_{n}:=\lvC<X_0,X_1>/I^{n+1}, 
\]
i.e. the algebra comprising linear combinations of words in the $X_{j}$ of length less than or equal to $n.$ The inverse limit of the $U_{n}$ is the power series algebra in the non-commuting variables
\[
U:=\lim_{\leftarrow}U_{n} = \lvC\!\ll \! X_{0},  X_{1}\!\gg.
\]
Now we set
${\sU}_{n}: = U_{n}\otimes {\sO}_{\Pomt}$ and
${\sU}:={{\lim}_{\leftarrow}}{\sU}_{n}.$ With the $\omega_{j}$ defined as above for $j=0,1,$ and $|w|$ denoting the length of the word $w$ in the $X_{j},$ a compatible family of connections on the ${\sU}_{n}$ can be defined, giving rise to a connection on ${\sU}$: Indeed, let 
\[
\sum_{|w|\leq n}f_{w}[w] \in {\sU}_{n}
\]
be arbitrary, and set
\[
\nabla_{n}(\sum_{|w|\leq n}f_{w}[w])=\sum_{|w|\leq n}df_{w}[w]-pr_{n}\sum_{|w|\leq n}f_{w}\sum_{i=1}^{m}\omega_{i}[wX_{i}]
\]
where $pr_{n}$ is the projection to ${\sU}_{n}$ - i.e. the augmented words $[wX_{i}]$ having length greater than $n$ are disregarded.
One checks readily that $\nabla_{n}$ is a connection on ${\sU}_{n},$ 
which is  unipotent (that is to say, a successive extension of trivial bundles $({\sO}_{\Pomt}^r, d)$; for a similar computation see 
\cite{Kim:UAM}).
Moreover, for $k >0$  the (suitably interpreted) restriction of the connection on ${\sU}_{n+k}$ to $ {\sU}_{n}$ evidently agrees with $\nabla_{n}.$ 
Hence $({\sU}, \nabla)$ is the inverse limit of unipotent connections on $X$. 

$\nabla$ is identical to the formal {\sc Knizhnik-Zamolodchikov} (KZ) equation
\[
{d}G(z, X_0,X_1) = \left(\frac{dz}{z}X_0 + \frac{dz}{1-z}X_1\right)G(z, X_0,X_1).\]
A fundamental solution to this equation asymptotic to $\exp(X_0\log z)$ as $z$ approaches 0 is the polylogarithm generating function $\Li(z, X_0, X_1)$, given by the {\sc Chen} series
\[
\Li(z, X_0,X_1):=\sum_{w}\int_{[\ora{01},z]}\omega_{i_1} \ldots \omega_{i_k}X_{i_1} \ldots X_{i_k}
\]
where $[\ora{01},z]$ denotes a tangential path from $\ora{01}$ to $z$ which winds around neither 1 nor $\infty$ in $\Pomt$; and other notation is as in the introduction.

\subsubsection{The reduced action on sections of $\sU$}
\label{S:red.act}
The $\Lambda$ action on $\Pomt$ by linear fractional transformations lifts to the (global) sections of $\sO_{\Pomt}$ in the obvious way. 
This produces an action on sections of $\sU$ once 
a suitable action of $\Lambda$ on the formal variables $X_0$ and $X_1$ is defined. The latter was determined by {\sc Okuda} and {\sc Ueno} in \S 3 of \cite{OU}, in which formal algebraic arguments and the theory of differential equations was used to compute the $\Lambda$ action on the fundamental solutions to the {\sc KZ} equation with specific asymptotics at 0, 1 and $\infty$ respectively, generalizing a calculation of {\sc Drinfel'd}. The action on $X_0$ and $X_1$ arises from a simple substitution action on the {\sc KZ} equation: 
\begin{e.g.}
Consider the element $\ol\sigma :z \mapsto 1-z$ of $\Lambda.$ Making this substitution in the {\sc KZ} equation yields
\[
-\frac{d}{dz} G(1-z, X_0,X_1)
=
\left(
\frac{X_0}{1-z} + \frac{X_1}{z}
\right)G(1-z,X_0,X_1)
\]
- i.e.
\begin{equation}
\label{e:s_1KZ}
\frac{d}{dz} \tilde{G}(z, X_0,X_1)
=
\left(
\frac{-X_1}{z} + \frac{-X_0}{1-z}
\right)\tilde{G}(z,X_0,X_1).
\end{equation}
This equation is identical to the original {\sc KZ} equation but for the interchanging of $X_0 \leftrightarrow -X_1.$ Therefore we define the action of $\ol\sigma$ on the pair $(X_0,X_1)$ 
of formal non-commuting variables,  
as the involution $(X_0, X_1) \mapsto (-X_1, -X_0)$.  
\end{e.g.}

This example may be imitated for each element of $\Lambda$, 
 %
and as in (25) of \cite{OU} it is convenient to summarize all transformations of $(X_0,X_1)$ which arise in this way.  The associated linear fractional transformations of $\Pomt$ are also tabulated:
\begin{center}
\begin{tabular}{rll}
Elt. of $SL(2, \lvZ/2\lvZ)$&
Lin. frac. tr.
&
Action on $(X_0, X_1)$
\\
1 :& $z \mapsto z$ & $(X_0, X_1) \mapsto (X_0,X_1)$
 \\
$\ol\sigma           : $         &$z \mapsto 1-z$  &$ (X_0, X_1) \mapsto  ( -X_1, -X_0)$ \\
$\ol\rho : $   &$ z \mapsto \frac{z}{z-1}$&$ (X_0, X_1) \mapsto (X_0, X_0-X_1)$ 
\\
$\ol\sigma \circ \ol\rho :$&$ z\mapsto \frac{1}{1-z}$
& 
$(X_0, X_1) \mapsto (-X_1,X_0-X_1)$ \\
$\ol\rho \circ\ol\sigma:$ & $z \mapsto \frac{z-1}{z} $ & 
$(X_0, X_1) \mapsto (X_1-X_0,-X_0)$\\
$\ol\rho\circ\ol\sigma\circ\ol\rho
= \ol\sigma \circ \ol\rho \circ \ol\sigma:$
&$z \mapsto \frac{1}{z}$& $(X_0, X_1) \mapsto (X_1-X_0,X_1)$ \\
\end{tabular}
\end{center}

Now one can state the
\begin{defn}
\label{d:SL2.act}
For every $\ol\ups\in SL(2, \lvZ/2\lvZ)$ and every global section $L(z, X_0,X_1)$ of $\sU$, set
\[
L^{\ol\ups}(z,X_0,X_1):=
L(\ol\ups(z),  \ol\ups X_0, \ol\ups X_1)
\] 
and refer to this as the $SL(2, \lvZ/2\lvZ)$-action on global sections of $\sU.$
\end{defn}

\begin{e.g.} 
\label{eg:2}
We compute $\Li^{\ol\sigma}(z, X_0,X_1):$ 
By construction, 
\[
\Li^{\ol\sigma}(z,  \ol\sigma X_0,\ol\sigma X_1) = \Li(1-z, X_0,X_1)
\]
is a fundamental solution to (\ref{e:s_1KZ}). Formally, $\Li(z, -X_1,-X_0)$ is also.
Recall from \S \ref{sss:UPB} that
\[
\Li(z,X_0,X_1) \exp(-X_0 \log z)\rightarrow 1
\]
as $z \rightarrow 0.$ Hence
\begin{equation}
\label{e:Li.asy}
\Li(z, -X_1,-X_0) \exp(X_1 \log z) \rightarrow 1
\end{equation}
as $z \rightarrow 0.$ Now recall from \cite{Cartier}
\[
\lim_{z \rightarrow 1} \Li(z, X_0,X_1) \exp(X_1 \log (1-z)) 
= 
\Phi_{KZ}(X_0,X_1), 
\]
where $\Phi_{KZ}(X_0,X_1)$ denotes the {\sc Drinfel'd} associator{\footnote{This expression can be taken as the definition of $\Phi_{KZ}$, but this formal power series can also be given more explicitly. See \cite{LeMurakami}.}},
or equivalently, 
\[
\lim_{z \rightarrow 0} \Li(1-z, X_0,X_1) \exp(X_1 \log z) 
= 
\Phi_{KZ}(X_0,X_1).
\]
But then $\Phi_{KZ}(X_0,X_1)\Li(z, -X_1,-X_0)$ and $\Li(1-z, X_0,X_1)$ share the same asymptotics near zero and both are solutions to the {\sc KZ} equation. By uniqueness of such solutions, then 
\[
\Li^{\ol\sigma}(z, X_0,X_1) = \Li(1-z, -X_1,-X_0) = \Phi_{KZ}(-X_1,-X_0) \Li(z, X_0,X_1)
\]
\end{e.g.}
  
We remark that by the symmetry in the above computation, it is evident that $\Phi_{KZ}(X_0,X_1)^{-1} = \Phi_{KZ}(-X_1, -X_0) 
,$ a fact which will be used often in what follows.

With the notation of \ref{d:SL2.act}, the computations of Propostion 2 of \cite{OU}, (which run in the same vein as \ref{eg:2}), may be summarized by
\begin{prop}
\label{p:OU}
\begin{eqnarray*}
\Li^{\ol\sigma}(z, X_0,X_1) 
&= &
\Phi_{KZ}(-X_1,-X_0) \Li(z, X_0,X_1)
\\
\Li^{\ol\rho}(z,  X_0,X_1) 
&=&
 \exp(\mp X_0 i \pi ) \Li(z, X_0,X_1)
\\
\Li^{\ol\sigma\circ\ol\rho}(z, X_0,X_1)
&=&
\exp(\pm X_1 i \pi ) \Phi_{KZ}(-X_1,-X_0)\Li(z,X_0,X_1)
\\
\Li^{\ol\rho\circ\ol\sigma}(z, X_0, X_1)
&=&
\Phi_{KZ}(X_0,X_0-X_1)^{-1}
\exp(\mp X_0 i \pi)  
  \Li(z,X_0,X_1)
  \end{eqnarray*}
and
\begin{eqnarray*}
&& 
\Li^{\ol\rho \circ\ol\sigma \circ \ol\rho }(z, X_0, X_1)\\
&=&
\exp(\pm(X_0 -X_1)i \pi ) \Phi_{KZ}(X_1-X_0,-X_0)
\exp(\mp X_0 i \pi) \Li(z, X_0,X_1)
\\
&=&
\Li^{\ol \sigma \circ\ol\rho \circ\ol\sigma }(z, X_0, X_1)
\\
&=&
\Phi_{KZ}(X_1-X_0,X_1) \exp(\pm X_1 i \pi )\Phi_{KZ}(-X_1,-X_0)\Li(z,X_0,X_1)
\end{eqnarray*} 
where the ambiguity in sign is according to $z$ being in the upper or lower half plane respectively.  
\end{prop}
The ambiguity in sign will be resolved in lifting the action to $PSL(2,\lvZ).$  

We remark that the equality 
 $\Li^{\ol\rho \circ\ol\sigma \circ \ol\rho }
 =\Li^{\ol \sigma \circ\ol\rho \circ\ol\sigma }$ follows from the well-definedness of the $\Lambda$ action and is a means of using the braid relation $\ol\rho \circ\ol\sigma \circ \ol\rho =  \ol \sigma \circ\ol\rho \circ\ol\sigma $ to establish the (highly non-trivial) hexagonal relations of {\sc Drinfel'd}, to wit 
 \begin{eqnarray*}
&& \Phi_{KZ}(X_1-X_0,X_1) \exp(\pm X_1 i \pi )\Phi_{KZ}(-X_1,-X_0)
\\
&=&
 \exp(\pm(X_0 -X_1)i \pi ) \Phi_{KZ}(X_1-X_0,-X_0) \exp(\mp X_0 i \pi).
 \end{eqnarray*}

\subsubsection{Lifting the action on sections of $\sU$ to $PSL(2, \lvZ)$}
\label{ss:lift.act}

The action of $SL(2, \lvZ/2\lvZ)$ on the formal variables $X_0$ and $X_1$ as given in the table in \S \ref{S:red.act} extends by linearity to polynomials in the $X_j$ with complex coefficients, and thereby to the quotients
\[
\lvC<\!\!X_0,X_1\!\!>/ I^{N+1}
\]
(where $I=(X_0,X_1)$ denotes the augmentation ideal); and hence to the inverse limit $\lvC\ll\!\!X_0,X_1\!\!\gg$. More precisely, we have:
\begin{defn} 
\label{d:ps}
The action of $\ol\alpha\in SL(2, \lvZ/2\lvZ)$ on a formal power series $F(X_0,X_1) \in \lvC\!\ll\!\!X_0,X_1\!\!\gg$ is given by 
\[
F(X_0,X_1)^{\ol\alpha}
:=
F( \ol\alpha X_0,\ol\alpha X_1).
\] 
\end{defn}


A given element $\alpha \in PSL(2, \lvZ)$ then acts on power series by reduction to $SL(2, \lvZ/2\lvZ)$. In this case we replace $\ol\alpha$ by $\alpha$ in the notation for the above action -i.e. we set
\begin{equation}
\label{e:SL2}
F(X_0,X_1)^{\alpha}:=F(X_0,X_1)^{\ol\alpha}.
\end{equation}

Now let $V_a$ be some open neighbourhood of $a \in \{0,1,\infty\}$ in $\mathbb{P}^{1}$ for which $(V_a \setminus\{a\})\cap\{0,1,\infty\}$ is empty. Then set 
$U_a:=V_a \setminus\{a\}.$ This is an open set of $\Pomt.$
Suppose that $L_a(z, X_0,X_1)$ is a section of $(\sU, \nabla)$ defined over $U_a$ - i.e. $L_a(z, X_0,X_1) \in \Gamma(U_a,\sU)$. 
As above, let $\ol\ups$ denote the image of $\ups\in PSL(2, \lvZ)$ under the usual projection map to $SL(2, \lvZ/2\lvZ).$ By identifying the elements of $PSL(2, \lvZ)$ with those of $G_{\ora{ab}}$ as in \S \ref{sss:top}, the $SL(2, \lvZ/2\lvZ)$-action on section of $\sU$ as in Definition \ref{d:SL2.act} can be lifted to an action of $PSL(2, \lvZ)$ as follows:
\begin{defn} Fixing a choice of basepoint $\ora{ab},$ then given $\ups \in PSL(2, \lvZ),$ analytically continue $L_a(z,\ol\ups X_0,\ol \ups X_1)$ along any path of the corresponding homotopy class of paths $\Psi^{-1}_{ab}(\ups)$ in $G_{\ora{ab}}$. 
This is an action of $PSL(2, \lvZ)$ for which the image of $L_a(z, X_0,X_1)$ will be denoted by
\[
L_a^{\upsilon}(z, X_0,X_1).
\]
\end{defn}

We proceed to compute this action for the distinguished section $\Li(z, X_0,X_1) \in \Gamma(U_{0}, \sU).$ Begin by setting
\[
\omega(X_0,X_1):=
\frac{dz}{z}X_0+\frac{dz}{1-z}X_1
\]
and write
\[
\omega(X_0,X_1)^{\alpha} = \omega(\ol\alpha X_0, \ol\alpha X_1 ).
\]
The following results will prove to be essential. The proofs are elementary.
\begin{lemma}
\label{l:pullback}
For any $\alpha \in PSL(2, \lvZ),$
\[
{\ol\alpha}^{*}\omega(X_0, X_1)
:=
\left[\ol\alpha^{*}\left(\frac{dz}{z}\right)\right]X_0
+
\left[\ol\alpha^{*}\left(\frac{dz}{1-z}\right)\right]X_1
=
\omega(X_0, X_1)^{\ol\alpha^{-1}}.
\]
\end{lemma}
\begin{lemma}
\label{l:pullback1}
For any $\alpha, \beta \in PSL(2, \lvZ),$
\[
{\ol\alpha}^{*}[{\ol\beta}^{*}\omega( X_0,  X_1)]
=
\omega( X_0,  X_1)^{(\ol\alpha\circ\ol\beta)^{-1}}=(\ol\alpha \circ \ol\beta)^{*}\omega(X_0,X_1).
\]
\end{lemma}

Now let $\int_\alpha\tilde\omega^n$ denote the $n$-fold {\sc Chen} iterated integral of the form $\tilde\omega$ along $\alpha$ - i.e. an iterated integral in which $\tilde\omega$ is repeated $n$ times. Also write $\int_\alpha\tilde\omega^0 =1.$ Then we have
\begin{prop}
\label{t:PSL.act}
For any $\alpha \in PSL(2, \lvZ),$
\[
\Li^{\alpha}(z, X_0,X_1)
=
F_{\alpha}(X_0,X_1)
\Li(z,  X_0, X_1)
\]
where $F_{\alpha}(X_0,X_1)$ is a formal power series given by the {\sc Chen} series
\[
F_{\alpha}(X_0,X_1)
:=
\sum_{n\geq0}\int_{\alpha}\omega( \ol\alpha X_0,\ol\alpha X_1)^n.
\]
\end{prop}
Implicitly here, $PSL(2, \lvZ)$ is identified with $G_{\ora{01}}$. 
\begin{proof}
When the {\sc Chen} iterated integrals are suitably interpreted - regularizing $\frac{dz}{z}$ at $z=0$ and $\frac{dz}{1-z}$ at $z=1$ in the usual way (cf. \cite{Joyner:ii}) -   for $z \not \in (-\infty, 0) \cup (1, \infty)$ the polylogarithm generating function may be expressed as
\[
\Li (z, X_0,X_1) =\sum_{n\geq 0}\int_{[\ora{01},z]}
\omega(X_0,X_1)^{n}
\]
with notation as above.

Now consider $\alpha \in G_{\ora{01}}$. Denote the endpoint thereof by $\ora{cd},$ and a path from $\ora{cd}$ to $\ol\alpha z$ which does not cross the real axis by $[\ora{cd}, \ol\alpha z].$
Then composing paths in the order as written (i.e. not the functional order), the analytic continuation of $\Li(z, X_0,X_1)$ along $\alpha$ is given by 
\begin{equation*}
\sum_{n\geq 0}\int_{\alpha\cdot[\ora{cd},\ol\alpha(z)]}\omega(X_0,X_1)^{n}.
\end{equation*}

Consider a typical integral which appears here.
Using the coproduct formula for iterated integrals (since $\omega(X_0, X_1)$ is a 1-form),
\begin{align*}
&
\int_{\alpha\cdot [\ora{cd},\ol\alpha z]}
\omega( X_0, X_1)^{n}
\\
&=
 \sum_{k=0}^{n}
\int_{\alpha}\omega( X_0, X_1)^{k}\cdot
\int_{[\ora{cd},\ol\alpha z]}
\omega(X_0, X_1)^{n-k}.
\end{align*}
Now
\begin{align*}
\int_{[\ora{cd},\ol\alpha z]}
\omega(X_0,X_1)^{n-k}
&=
\int_{[\ora{01}, z]}
\ol\alpha^{*}\omega(X_0, X_1)^{n-k}\\
&=
\int_{[\ora{01},z]}
\omega(\ol\alpha^{-1} X_0 , \ol\alpha^{-1}X_1 )^{n-k}
\end{align*}
by the Lemma \ref{l:pullback}. Hence, replacing $X_j$ by $\ol\alpha X_j$ for $j=0,1$,  $\Li^{\alpha}(z, X_0,X_1)$ is the same as
\begin{align*}
&
\sum_{n \geq 0}
\sum_{k=0}^{n}
\int_{\alpha}\omega(\ol\alpha X_0, \ol\alpha X_1)^{k}\cdot\int_{[\ora{01},z]}
\omega(\ol\alpha^{-1}\circ \ol\alpha X_0,\ol\alpha^{-1}\circ\ol\alpha X_1)^{n-k}
\\
&=\left[
\sum_{n \geq 0}\int_{\alpha}\omega(\ol\alpha X_0,  \ol\alpha X_1)^{n}\right]
\cdot
\Li(z, X_0, X_1).
\end{align*}
\end{proof}

Now we can prove the fundamental 
\begin{cor}
\label{c:qtqH}
\[
\Li^{\rho}(z,X_0,X_1)
= 
\exp(i \pi X_0)\Li(z, X_0, X_1)
\]
\[
\Li^{\sigma}(z, X_{0}, X_{1}) = \Phi_{KZ}(X_{0},X_{1})^{\sigma}\Li(z, X_0, X_1)
\]
\end{cor}
\begin{proof}
Recall that $\Psi(r) = \rho$ and $\Psi(s) = \sigma$ in the notation of \S \ref{sss:top}. Also, $\ol\rho(X_0,X_1) = (X_0,X_0-X_1)$ while $\ol\sigma(X_0,X_1) = (-X_1,-X_0).$

One computes
\begin{equation*}
\int_{r}\left(\frac{dz}{z}X_0+ \frac{dz}{1-z}X_1\right)^{n}
=\frac{X_0^n}{n!}\left(\int_r\frac{dz}{z}\right)^n
= \frac{(i\pi X_0)^{n}}{n!}
\end{equation*}
by repeatedly using the shuffle product for iterated integrals and the fact that the integrals in which $\frac{dz}{1-z}$ occur vanish along $r$. 
Thence
\begin{align*}
F_{\rho}(X_0,X_1)
&=
\sum_{n=0}^{\infty}\int_{r}\left(\frac{dz}{z} X_0+ \frac{dz}{1-z}(X_0-X_1)\right)^{n}
\\
&=
\sum_{n=0}^{\infty}\frac{(i \pi X_{0})^{n}}{n!}
\\
&=
\exp (i \pi X_0).
\end{align*}

It is well-known that
\[
\Phi_{KZ}(X_0,X_1)
=
\sum_{n=0}^{\infty}\int_{[0,1]}
\left(\frac{dz}{z}X_0 +\frac{dz}{1-z}X_1\right)^{n},
\]
in which expression we understand the integrals to be regularized at 0 and 1 as before, and the shuffle regularization of iterated integrals is applied to those terms which otherwise diverge (see 3.4 of \cite{Cartier} for example). From Proposition \ref{t:PSL.act} the second assertion of the corollary follows.
\end{proof}

The power series which arise here do not look too different from those which result in the case of the $SL(2, \lvZ/2\lvZ)$ action as in \S \ref{S:red.act}. Where $\sigma$ is concerned,  the reason for this is that $\ol\sigma$ has a unique lift to $PSL(2,\lvZ)$. On the other hand, unlike that of $\ol\rho$ the action of $\rho$ is not involutive. 

\begin{thm}
\label{t:cocycle}
$F_{\bullet}(X_0,X_1)$ is an injective 1-cocycle for $PSL(2, \lvZ)$ in the multiplicative group of formal power series in the non-commuting variables $X_0$ and $X_1$ equipped with the action of $PSL(2,\lvZ)$ factoring through that of $SL(2, \lvZ/2\lvZ)$. Specifically, for any $ \ups, \ups' \in PSL(2, \lvZ),$
\[
F_{\ups}(X_0,X_1)^{\ups'}F_{\ups'}(X_0,X_1)= F_{\ups'\circ\ups}(X_0,X_1).
\]
\end{thm}
\begin{proof}
Consider such arbitrary $\ups$ and $\ups' \in PSL(2, \lvZ)$ and identify them with some choices of paths in the corresponding homotopy classes of $G_{\ora{ab}}$.
Also, interpret $\int_{v}\omega^{0}$ as 1 so that
\[
F_{\ups' \circ \ups}(X_0,X_1)
=
\left.
\sum_{n=0}^{\infty}
\int_{\ups\ups'}
\left(\frac{dz}{z}Y_0+\frac{dz}{1-z}Y_1\right)^{n}\right|_{Y_j =\ol\ups' \circ \ol\ups X_j\;:j=0,1}
\]
Now recall from \S \ref{sss:top} that $\ups\ups' = \ups \cdot [\ups]*\ups'$, write $Y_j=\ol\ups' \circ \ol\ups X_j$ for $j=0,1$ as above, and use the coproduct formula for iterated integrals to compute
\begin{align*}
\int_{\ups\ups'}
\left(\frac{dz}{z}Y_0+\frac{dz}{1-z}Y_1\right)^{n}
&=
\sum_{k=0}^{n}
\int_{\ups}\omega(Y_0,Y_1)^{k}
\int_{[\ups]*\ups'}\omega(\ol\ups' \circ \ol\ups X_0,\ol\ups' \circ \ol\ups X_1)^{n-k}
\\
&=
\sum_{k=0}^{n}
\int_{\ups}\omega(Y_0,Y_1)^{k}
\int_{\ups'}\ol\ups^{*}\omega(\ol\ups' \circ \ol\ups X_0,\ol\ups' \circ \ol\ups X_1)^{n-k}
\\
&=
\sum_{k=0}^{n}
\int_{\ups} \omega(Y_0,Y_1)^{k}
\int_{\ups'}\omega((\ol\ups' \circ \ol\ups) \circ \ol\ups^{-1}X_0, (\ol\ups' \circ \ol\ups) \circ \ol\ups^{-1}X_1)^{n-k}
\end{align*}
using Lemmas \ref{l:pullback} and \ref{l:pullback1}. Hence
\begin{align*}
F_{\ups' \circ \ups}(X_0,X_1)
&=
\sum_{n \geq 0}\sum_{k=0}^{n}
\int_{\ups} \omega(\ol\ups' \circ \ol\ups X_0, \ol\ups' \circ \ol\ups X_1)^{k} \int_{\ups'}\omega(\ol{\ups'} X_0, \ol{\ups'} X_1)^{n-k}
\\
&=
F_{\ups}(X_0,X_1)^{\ups'}F_{\ups'}(X_0,X_1).
\end{align*}

Finally we prove the injectivity: Consider any $\alpha \in PSL(2, \lvZ)$ for which $F_{\alpha}(X_0,X_1)=1.$ We show that such $\alpha$ is necessarily trivial. First observe that because $\ol\alpha$ is invertible, also $F_{\alpha}(\ol\alpha X_0,\ol\alpha X_1)=1.$ But then 
\[
\int_{\alpha}\omega( X_0, {X_1})^{n}
=0
\]
for each $n\geq 1,$ since each such integral expression is homogeneous of degree $n$ in the $X_j.$ Even further, the coefficients of the monomials $X_{i_1} \ldots X_{i_n}$ here are all of the form of 
\[
\int_{\alpha} \omega_{i_1} \ldots \omega_{i_n}
\]
where $i_j$ is either 0 or 1. Consequently, all such integrals are necessarily zero. But then by {\sc Chen's} $\pi_1$ {\sc De Rham} Theorem (given for this case in Theorem 10 in\cite{Hain:Lec}), necessarily $\alpha$ is trivial, as was to be shown.
\end{proof}

Let $\Gamma$ denote an arbitrary fixed subgroup of $PSL(2, \lvZ)$ and define
\[
\sF_{\Gamma}:=\{F_{\alpha}(X_0,X_1) \in \Cs |\alpha \in \Gamma\}.
\]
\begin{lemma}
\label{l:LE}
The elements of $\sF_{PSL(2, \lvZ)}$ are group-like.
\end{lemma}
\begin{proof}
$\Phi_{KZ}(X_0,X_1)$ is group-like by construction. $e^{i \pi X_0}$ is group-like since $X_0$ is primitive. In fact, replacing $(X_0,X_1)$ in each of these formal series by any pair of primitive elements of $\lvC<X_0,X_1>$, the resulting series are also group-like. Now the images of $X_0$ and $X_1$ under the action of the elements of $SL(2, \lvZ/2\lvZ)$ are all primitive. Consequently, each $F_\alpha(X_0,X_1)$ is a product of group-like elements, making it group-like too, since the {\sc Lie} exponentials form a group.
\end{proof}

Endow $\sF_{\Gamma}$ with a multiplication $\circledast$ coming from the $SL(2, \lvZ/2 \lvZ)$ action - i.e. set 
\[
F_\beta(X_0,X_1) \circledast F_\alpha(X_0,X_1) 
:= 
F_\alpha(X_0,X_1)^{\beta}F_{\beta}(X_0,X_1) = F_{\beta \circ \alpha}(X_0,X_1).
\]
This is well-defined by \ref{t:cocycle}.
 Also from \ref{t:cocycle} one obtains
\begin{thm}
\label{t:torsor}
$(\sF_{\Gamma}, \circledast)$ is a group which is isomorphic to $(\Gamma, \circ)$.
\end{thm}

\subsection{Monodromy of polylogarithms}
Identifying any $g \in \pi_1(\Pomt, \ora{01})$ with $\gamma \in PSL(2, \lvZ)$ via (\ref{e:ses}), as made explicit in \S \ref{sss:top}, 
the monodromy of $\Li(z, X_0,X_1)$ about $g$ is equal to $\Li^{\gamma}(z,X_0,X_1)$, since $\ol\gamma$ is trivial.  As a consequence of \ref{t:PSL.act} and \ref{t:cocycle}, determining the monodromy is now an easy calculation. For example, about the generators of the fundamental group, as first proven in \cite{MPvdH} by means of direct methods, we have
\begin{prop}
\label{p:mono}
The monodromy of $\Li(z, X_0,X_1)$ about the loop $r^2$ about 0 in $\Pomt$ based at $\ora{01}$ is given by
\[
\Li^{\rho^2}(z,X_0,X_1)
= 
\exp(2i \pi X_0)\Li(z,X_0,X_1),
\]
while that about the loop $sr^2s$ about 1 in $\Pomt$ is given by 
\[
\Li^{\sigma\rho^2\sigma}(z, X_{0}, X_{1}) = \Phi_{KZ}(X_0,X_1)^{\sigma}\exp(-2 \pi i X_1)
\Phi_{KZ}(X_{0},X_{1}) \Li(z, X_0, X_1).
\]
\end{prop}
\begin{proof}
Recall that the loop $r^2$ about 0 in $\Pomt$ based at $\ora{01}$ corresponds to $\rho^{2}$ in $PSL(2, \lvZ)$. Now  
\begin{align*}
\Li^{\rho^2}(z, X_0,X_1)
&=
F_{\rho}(X_0,X_1)^{\rho}F_\rho(X_0,X_1)\Li(z,  X_0,X_1)
\\
&=
\exp(\pi i X_0)^{\rho}\exp(\pi i X_0) \Li(z, X_0,X_1)
\\
&=
\exp(2 \pi i X_0) \Li(z, X_0,X_1)
\end{align*}
since ${\ol\rho} (X_0) = X_0$.

The loop about 1 in $\Pomt$ based at $\ora{01}$ is $sr^2s$ in former notation, corresponding to 
$\sigma \rho^2\sigma$ in $PSL(2, \lvZ).$ 
\begin{align*}
&\Li^{\sigma\rho^2\sigma}(z, X_{0}, X_{1}) 
\\
&=
F_{\sigma}(X_0,X_1)^{\sigma\rho^2}
F_{\rho}(X_0,X_1)^{\sigma\rho}
F_{\rho}(X_0,X_1)^{\sigma}
F_{\sigma}(X_0,X_1)\Li(z, X_0,X_1)
\\
&=
\Phi_{KZ}(X_0,X_1)
\exp(-\pi i X_1)
\exp(-\pi iX_1)
\Phi_{KZ}(X_1,X_0)^{\sigma}
\Li(z, X_0,X_1)\\
%
&
\;\;\;\text{from the definition of the respective actions of $\ol\sigma$ and $\ol\rho$ on $X_0,X_1.$}
\end{align*}
\end{proof}

\section{Application: proving the analytic continuation and functional equation of $\zeta(s)$}

As outlined in the introduction, data encoded in the modular action 
on $\Li(z, X_0,X_1)$ may be used to give a family of proofs of the analytic continuation and functional equation of the {\sc Riemann} zeta function. In particular, the analytic continuation may be effected (as in \S \ref{s:AC} below) using functional relations known as the {\sc Euler} connection formulae. As shown in Proposition 5 of \cite{OU}, these arise from equating coefficients of the respective sides of 
\[
\Li(1-z, -X_1,-X_0) = \Phi_{KZ}(-X_1,-X_0)\Li(z, X_0,X_1).
\]
This equation giving $\Li^{\ol\sigma}(z, X_0,X_1)$ is necessarily the same as 
\[
\Li^{\sigma}(z, X_0,X_1) = \Phi_{KZ}(-X_1,-X_1)\Li(z, X_0,X_1)
\]
since $\sigma$ and the reduction thereof are both involutions. 

Monodromy data of polylogarithm functions is used to prove the functional equation itself in \S \ref{fe}. As shown in Proposition \ref{p:mono} above, such information is given by an equation which is an easy consequence of the modular action computed in Corollary \ref{c:qtqH}.

The proofs are modifications of {\sc Riemann's} contour integral method, each based on a member of an infinite family of integral expressions for $\zeta(s)$ given in \S \ref{Intro}.


\subsection{Families of integral expressions for $\zeta(s)$}
\label{Intro}

In previous work, \cite{Joyner:ii}, the author developed a theory of complex iterated integral generalizing the usual notion of iterated integral as in the work of {\sc Chen}. In particular, on $\Pomt$, if $F(z)$ denotes a function with $F(0)=0$ having {\sc Taylor} series expansion on the unit disc for which the $n$th coefficient is $O(n^{k})$ for some $k\geq 0$, then we define
\begin{equation}
\label{defCII}
L[F](s,z)=\int_{[0,z]}F(t)\left(\frac{dt}{t}\right)^{s}
:=
\int_{0}^{z}\frac{(\log z-\log t)^{s-1}}{\Gamma(s)}F(t) \frac{dt}{t},
\end{equation}
where the usual regularization of the logarithm at zero is understood, in which case the integral can be shown to converge on $\Re s>k+1.$ In what follows
we shall write $L[F](s):=L[F](s,1)$ and say that the functions $F(z)$  satisfying the conditions given above are $k$-{\sc Bieberbach}.

This complex iterated integral turns out to coincide under the change of variables $x=-\log t$ with the fractional integral as defined by {\sc Riemann} and {\sc Liouville}; for which the {\em additive iterativity} property 
\begin{equation}
\label{addit}
L[F](s) = L\left[ \int_{[0,z]}F(t)\left(\frac{dt}{t}\right)^{w}\right](s-w)
\end{equation}
holds for those $w$ for which all relevant integrals converge. ($w$ should have $\Re w>k+1$ and $\Re (s-w) >k+1.$)
Although this much is classical, the iterated integral perspective lends itself to powerful generalization 
and has various number theoretic consequences (see \cite{Joyner:ii}), among them the non-classical {\em multiplicative iterativity} property
\[
\int_{[0,1]}\sum_{n=1}^{\infty}a_{n}t^{n}\left(\frac{dt}{t}\right)^{s}
=
\int_{[0,1]}\sum_{n=1}^{\infty}a_{n}t^{n^{k}}\left(\frac{dt}{t}\right)^{{s}/{k}}
\]
for positive integer $k$. 

Each of these respective iterativity properties gives rise to an infinite family of integral expressions for  the {\sc Riemann} zeta function $\zeta(s)$:

In $\Pomt$ coordinates {\sc Abel's} integral for $\zeta(s)$ becomes
\begin{equation}
\label{AI}
\zeta(s)
= 
\int_{[0,1]}\frac{t}{1-t}\left(\frac{dt}{t}\right)^{s}
=L\left[ \frac{t}{1-t} \right](s) 
\end{equation}
and hence by (\ref{addit}) may be expressed by any one of the family of integrals
\begin{equation}
\label{zetaPL}
\zeta(s) 
=
\int_{[0,1]}\Li_{\mu}(t)\left(\frac{dt}{t}\right)^{s-\mu}
\end{equation}
for integer $\mu$, with $\Li_{\mu}(z)$ denoting  the usual polylogarithm function 
\[
\Li_{\mu}(z):= \int_{[0,z]} \frac{t}{1-t}\left(\frac{dt}{t}\right)^{\mu}
\]
when $\mu >0$.
On the other hand, by multiplicative iterativity and use of {\sc Abel's} integral, for positive integer $k$
\[
\zeta(s) = 
\int_{[0,1]}\sum_{n=1}^{\infty}t^{n^{k}}\left(\frac{dt}{t}\right)^{{s}/{k}}.
\]
Observe that the case of $k=2$ corresponds to the theta function integral used in {\sc Riemann's} second proof of the functional equation of $\zeta(s)$, under the change of variables $t=e^{-\pi u}.$


Since {\sc Abel's} integral (\ref{AI})
(which forms the basis of {\sc Riemann's} first proof of the functional equation) belongs to both families of integrals it is interesting to exhibit a proof
of the functional equation 
making use of an integral which is a member of the additive family of integrals but not of the multiplicative family. To this we next proceed.


\subsection{Analytic continuation at integer parameter $\mu > 1$}
\label{s:AC}
Consider the dilogarithm integral
\[
\zeta(s) =\int_{[0,1]}\Li_{2}(z)\left(\frac{dz}{z}\right)^{s-2} =  \int_{0}^{\infty}\Li_{2}(e^{-x})\frac{x^{s-2}}{\Gamma(s-2)}\frac{dx}{x},
\]
(where we take $x=-\log z$ to obtain the last integral). For the analytic continuation of this integral to complex values of $s$ for which $\Re s \leq 2$ the standard ({\sc Hankel}) contour $C$ (as pictured below) 
is not suitable with this integrand, because the dilogarithm monodromy term arising from moving about 0 is $2 \pi i x,$ which does not have finite {\sc Mellin} transform. 
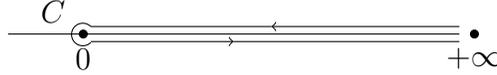
\begin{center}\begin{tikzpicture}
\node (a1) at (-0.4, 0.3) {$C$};
\draw [->] (5,0.1) -- (2.5,0.1);
\draw  (2.5, 0.1) -- (0.1, 0.1);
\draw (5,0) -- (-1, 0);
\draw (2, -0.1) -- (5, -0.1);
\draw [->] (0.1, -0.1) -- (2, -0.1);
\draw (0.1, 0.1) arc (45:315 : 0.15);
\filldraw (0,0) node [below=1.5pt] {0} circle (1.5pt);
\filldraw (5.2,0) node [below=1pt] {$+\infty$} circle (1.5pt);
\end{tikzpicture}\end{center}

Instead we use {\sc Euler's} dilogarithm inversion formula, which is a functional equation of the dilogarithm effecting a change between $z=0$ and $z=1.$ This shifts the monodromy of $\Li_{2}(e^{-x})$ to $x=+\infty$, so that {\sc Hankel's} contour may be used.

Now write $\Li_{2,1 \times (n-2)}(z) := \Li_{211\ldots 1}(z)$ with the index $1$ repeated
$n-2$ times; i.e. the multiple polylogarithm
\[
\Li_{2{{{11\ldots1}}}}(z) = \int_{[0,z]}\frac{dt}{1-t}\frac{dt}{t}{{\frac{dt}{1-t}
\frac{dt}{1-t} \cdots \frac{dt}{1-t}}}
\]
in which the form $\frac{dt}{1-t}$ occurs a total of $n-1$ times.
By using a generalized version of the dilogarithm inversion formula, we find in general:
\begin{thm}
For each integer $m \geq 2,$ 
\[
\zeta(s) = -\frac{(s-m)\Gamma(m)\Gamma(1-s)}{2\pi i}\int_{C}
\frac{\Li_{2,1 \times (m-2)}(1-e^{-x})}{x^{m}}(-x)^{s}\frac{dx}{x}
\]
for all complex $s \neq 1$ satisfying $\Re s<m.$
\end{thm}

For each fixed $m$, together with (\ref{zetaPL}) this result gives the analytic continuation of $\zeta(s)$ to all values other than $\{1\}\cup\{\Re s=m \}$.

\begin{proof} Throughout let $m$ denote an integer with $m \geq 2.$ Then
\begin{align*}
\zeta(s)
&=
\int_{0}^{1}\Li_{m}(e^{-x})\frac{x^{s-m}}{\Gamma(s-m)}\frac{dx}{x} 
+ 
\int_{1}^{\infty}\Li_{m}(e^{-x})\frac{x^{s-m}}{\Gamma(s-m)}\frac{dx}{x}\\
&=
\int_{0}^{1}\left[ \frac{\Li_{m}(e^{-x})}{x} - \frac{\zeta(m)}{x}\right]
\frac{x^{s-m}}{\Gamma(s-m)}{dx}
+
\frac{\zeta(m)}{(s-m)\Gamma(s-m)} 
+ \int_{1}^{\infty}\Li_{m}(e^{-x})\frac{x^{s-m}}{\Gamma(s-m)}\frac{dx}{x}.
\end{align*}
This last expression holds also for $m-1 < \Re s <m$ by analytic continuation, since $\Li_{m}(e^{0})=\zeta(m).$ But on this vertical strip, 
\[
\frac{1}{s-m}=-\int_{1}^{\infty}x^{s-m-1}dx.
\]
Consequently, provided that $m-1< \Re s < m,$ we can write
\begin{equation}
\label{Ti}
\zeta(s) = \int_{0}^{\infty}\left(\frac{\Li_{m}(e^{-x}) - \zeta(m)}{x}\right)
\frac{x^{s-m}}{\Gamma(s-m)}dx.
\end{equation}
(This much is patterned on a similar analytic continuation in \cite{Titch}.)

Now from  (45) in \cite{OU} using exponential coordinates and $\Li_{1}(z) = -\log(1-z)$,  {\sc Euler's} connection formula takes the form of
\[
\Li_{m}(e^{-x}) - \zeta(m) = -\Li_{2,1\times(m-2)}(1-e^{-x})
-\frac{x^{m-1}}{(m-1)!}(-\log(1-e^{-x}))
\]
\[
-\frac{x^{m-2}}{(m-2)!}\Li_{2}(e^{-x})
-
\ldots -\frac{x^{2}}{2!}\Li_{m-2}(e^{-x})
 - x\Li_{m-1}(e^{-x}).
\]
In substituting this expression into (\ref{Ti}) we notice immediately that all other polylogarithm integral expressions as in (\ref{zetaPL}) 
with $\mu=1, \ldots, m-1$ appear. These expressions are also valid for $m-1<\Re s <m$, so in each case, the resulting expression may be replaced by some multiple of $\zeta(s).$ We show this explicitly when $m=2:$ Then,
\begin{equation}
\label{0}
\zeta(s) = \int_{0}^{\infty} \left(\frac{\log(1-e^{-x})}{x} - \frac{\Li_{2}(1-e^{-x})}{x^{2}}\right)\frac{x^{s-1}}{\Gamma(s-2)}{dx}
\end{equation}
whenever $1<\Re s<2.$ Here, for $\Re s >1$ we have
\[
\zeta(s) = \int_{[0,1]}\Li_{1}(z)\left(\frac{dz}{z}\right)^{s-1} = -\int_{0}^{\infty}\log (1-e^{-x})\frac{x^{s-1}}{\Gamma(s-1)}\frac{dx}{x}.
\]
From  $(s-2)\Gamma(s-2) = \Gamma(s-1)$ it then follows that
\begin{equation}
\label{1}
\int_{0}^{\infty}\frac{\log (1-e^{-x})}{x}\frac{x^{s-1}}{\Gamma(s-2)}dx = -(s-2)\zeta(s),
\end{equation}
for $1<\Re s<2.$ Then using (\ref{1}) in (\ref{0}), 
\[
(s-1)\zeta(s)  = -\int_{0}^{\infty}\frac{\Li_{2}(1-e^{-x})}{x^{2}}\frac{x^{s-1}}{\Gamma(s-2)}dx.
\]

Most generally, repeated use of the functional equation 
\begin{equation}
\label{Gfe}
\Gamma(r+1) =r\Gamma(r)
\end{equation}
together with (\ref{zetaPL}) shows that for each integer $k$ with $1 \leq k \leq m-1$,
\begin{align*}
-\int_{0}^{\infty}\frac{x^{m-k}\Li_{k}(e^{-x})}{(m-k)!x}\frac{x^{s-m}}{\Gamma(s-m)}dx
& = -\frac{(s-m) \ldots (s-k-1)}{(m-k)!}\zeta(s) 
\\
&= -
\left(\begin{array}{c}
s-k-1 \\ m-k\end{array}
\right)\zeta(s).
\end{align*}
Adding the negative of such expressions to both sides of our equation for $\zeta(s)$  (found by substitution of the {\sc Euler} connection formula into (\ref{Ti})) and using a simple inductive argument to add up the terms of the coefficient, (adding first $1+(s-m)$ to obtain $s-m+1$ then taking this as a common factor in summing with the next term and so on), the left side becomes
\[
\left[1+ \left(\begin{array}{c}
s-m \\ 1
\end{array}
\right)
+ \ldots + \left(\begin{array}{c}s-2 \\ m-1
\end{array}\right)
\right]
\zeta(s)
=
\left(\begin{array}{c}
s-1 \\ m-1
\end{array}\right)\zeta(s)
\]
while the right side is given by
\[
-\int_{0}^{\infty} \Li_{2,1\times (m-2)}(1-e^{-x})\frac{x^{s-m}}{\Gamma(s-m)}
\frac{dx}{x}.
\]

By equating coefficients of the first equation of \ref{p:mono}, one sees that $\Li_{2, 1 \times(m-2)}(z)$ has no monodromy about $z=0$. Thus, $\Li_{2, 1\times(m-2)}(1-e^{-x})$ has no monodromy about $x=0.$ Now with $C$ as above, consider 
\[
I(s):= \int_{C}\frac{\Li_{2, 1\times(m-2)}(1-e^{-x})}{x^{m}}(-x)^{s-1}{dx}.
\]
Here the branch cut for the logarithm is taken along the negative real axis, so that for the portion of $C$ above the real axis from $x=+\infty$ to $x=0$, 
\[
(-x)^{s} = e^{s(\log x -i\pi)}
\]
and along the part of $C$ below the real axis back from 0 to $+\infty,$
\[
(-x)^{s} = e^{s (\log x + i\pi)}.
\]

Now along the arc, say with $|x| = \varepsilon$, which is the piece of $C$ around $x=0,$ the integrand is bounded by 
\[
M{\varepsilon}
\left|
x^{(\Re s) - m}\right|e^{2 \pi \varepsilon}
\] 
for some constant $M>0$
because $\Li_{2, 1\times(m-2)}(1-e^{-x})$ vanishes at $x=0$ at least to the same order as does $x.$ Now since $\Re s >m-1,$ the integral about $|x| = \varepsilon$ approaches 0 as $\varepsilon$ becomes very small. (The integral is of the order of $\varepsilon ^{\Re s-m+1}.$) 

Consequently, in the limit as $\varepsilon$ approaches 0, we have
\begin{align*}
I(s) &\rightarrow 
-e^{-i\pi s}\int_{\infty}^{0}\frac{\Li_{2, 1\times(m-2)}(1-e^{-x})}{x^{m}}
x^{s}\frac{dx}{x}\\
&\;\;\;
-
e^{i \pi s}\int_{0}^{\infty}\frac{\Li_{2,1\times(m-2)}(1-e^{-x})}{x^{m}}
x^{s}\frac{dx}{x}
\\
&=
(e^{i\pi s} -e^{-i \pi s}) \left(\begin{array}{c}s-1\\ m-1\end{array}\right)\Gamma(s-m)\zeta(s).
\end{align*}

Now $
2i\sin (\pi s) = 
e^{i\pi s} -e^{-i \pi s}$ and we may repeatedly use (\ref{Gfe}) to see that 
\[
\left(\begin{array}{c}s-1 \\ m-1\end{array}\right)\Gamma(s-m)
=
\frac{\Gamma(s)}{(s-m) \cdot (m-1)!}
=
\frac{\Gamma(s)}{(s-m) \Gamma(m)}.
\]
Moreover,
\[
\frac{\pi}{\sin{\pi s}} = \Gamma(s)\Gamma(1-s),
\]
so that 
\begin{align*}
(e^{i \pi s}- e^{-i \pi s})\left(\begin{array}{c}s-1 \cr m-1\end{array}\right)\Gamma(s-m)
&=
2i \sin(\pi s)\frac{\Gamma(s)}{(s-m)\Gamma(m)}
\\
&=
\frac{2 i \pi }{(s-m)\Gamma(m)\Gamma(1-s)}.
\end{align*}
Hence
\begin{equation}
\label{2}
\zeta(s) = \frac{(s-m)\Gamma(m)\Gamma(1-s)}{2 \pi i}\int_{C}
\frac{\Li_{2, 1\times(m-2)}(1-e^{-x})}{x^{m}}(-x)^{s-1}{dx}.
\end{equation}
This expression has been proven for $m >\Re s>m-1,$ but converges for all $s$ having $\Re s < m$ with the possible exception of the poles $s = 1, \ldots,  m-1$ of $\Gamma(1-s)$ in this region of the plane. (When $\Re s \geq m,$ the integrand exhibits unsuitable behavior at infinity.) However,
for $s= 2, \ldots,  m-1$ the integral vanishes by the usual argument: there is no monodromy about 0, so the integrals above and below the real line differ by a factor of $-1$ and approach the same absolute value in the limit as the contour approaches the real line; while near zero, one computes by {\sc L'H{\^o}pital's} Rule that
\[
\lim_{x\rightarrow 0}
\frac{\Li_{2,1\times(m-2)}(1-e^{-x})}{x^{m}}
=
\lim_{x\rightarrow 0}
\frac{(-1)^{m-1}}{m!x}.
\]
This leaves only the simple pole at $s=1,$ for which we may compute the residue using (\ref{2}). Firstly, 
\[
\Res_{s=1}\Gamma(1-s) = -1, 
\]
so
\begin{align*}
\Res_{s=1}\zeta(s) &= 
\lim_{s=1}(s-1)\Gamma(1-s)
\frac{2 \pi i}{(m-1)!(m-1)}\frac{(s-m)\Gamma(m)}{2\pi i}\\
&=
(-1) \frac{(1-m)(m-1)!}{(m-1)!(m-1)} = 1
\end{align*}
as is well known.

The analytic continuation for $\zeta(s)$ to $\Re s<m$ is achieved by (\ref{2}). 
\end{proof}

\subsection{Dilogarithm proof of the functional equation}\label{fe}
In principle, we may now imitate {\sc Riemann's} contour integral proof using each of the integrals of (\ref{2}). For each $m \geq 2$ this would be done by using the monodromy of $\Li_{2, 1\times (m-2)}(z)$, as may be  calculated using \ref{p:mono} 
(by equating the coefficients of terms of the respective power series).
The interesting aspect of the computation is that the monodromy terms coming from considering $\Li_{2,1\times(m-2)}(1-e^{-x})$ around $\infty$ have themselves monodromy about 
$2\pi i n$ for integer $n$, and this monodromy of the monodromy is what contributes terms that add up to $\zeta(1-s)$ multiplied by some factor.

All such proofs follow the same pattern, so for clarity of exposition, we restrict ourselves to the case that $m=2$ and present a careful proof in this situation:

Consider the punctured surface $X_{\log} \iso \lvC \backslash 2 \pi i \lvZ$ on which the logarithm function is single valued. This is the appropriate space to which to lift the contour $C$ in order to emulate {\sc Riemann's} calculation. 
However, it suffices for our purposes to replace this space by $\lvC$ with a logarithmic branch cut along the negative real axis. 
Now let $a_0\rightarrow a_1$ denote the straight line segment from $a_0$ to $a_1$ in $\lvC,$ and $a_0\rightarrow a_1\rightarrow\cdots \rightarrow a_n$ the succession of line segments from $a_0$ to $a_1$, $a_1$ to $a_2$ etc. on to $a_n$. If $P$ is a path and $u\in \lvC,$ write $P+u$ for the translation of $P$ by $u$. Next, fix small positive $\varepsilon$ and $\delta$ and real $R$ and $N$ with $R,N \gg 0.$ Denote the circular path of radius $\eta = \sqrt{\delta^2+\varepsilon^2}$ about 0 in the clockwise direction from $\delta-i\varepsilon$ to $\delta+i\varepsilon$ by $\gamma_0.$ Then set $C_0=(R-i\varepsilon \rightarrow \delta-i\varepsilon) \cdot \gamma_0 \cdot (\delta+i\varepsilon \rightarrow R+i\varepsilon)$ where $\;\cdot\;$ denotes concatenation of paths. Also, write $L_\infty$ for the line segment $R+i\varepsilon \rightarrow R+2\pi i  -i \varepsilon.$ Finally, let $C_{R, N}^{\delta, \varepsilon}$ denote the following closed path:
\[
C_0 \cdot L_{\infty} \cdot (C_0+2\pi i)\cdot (L_{\infty}+2\pi i)\cdots (C_0+2\pi i N) \cdot (R+2\pi iN+i \varepsilon \rightarrow R+2\pi iN+\pi i
\]
\[
\rightarrow -R+i\pi (2N+1)\rightarrow -R-i\pi(2N+1)
\rightarrow R-i\pi(2N+1) \rightarrow R-2\pi iN -i \varepsilon)
\]
\[
\cdot (C_0 - 2\pi iN)\cdot (L_{\infty}-2\pi i(N-1))\cdots (C_0-2\pi i)\cdot (L_{\infty}-2\pi i).
\]

Now  $\Li_{2}(1-e^{-x})$ has  the monodromy terms 
$2 \pi i{\log(1-e^{-x}})$  along the paths $R+2k\pi i\rightarrow R+2(k+1)\pi i$ (for integer $k$), since $\Li_{2}(z)$ has monodromy term $-2\pi i\log z$ moving in a positive direction around $z=1$; and as $x$ rises by $2 \pi i$ in $\lvC,$ $1-e^{-x}$ describes a negatively oriented circle about $x=1$. These logarithmic terms themselves have monodromy of $2\pi i$ about $2 k \pi i$ for integer $k$, because $1-e^{-x} = 0$ at such points, and for $\eta$ sufficiently small,  a negatively oriented loop of radius $\eta$ about $2 k\pi i$ maps to a positively oriented path encircling $x=0$ under $x \mapsto 1-e^{-x}$.

Hence, the region enclosed by such a path $C_{R,N}^{\delta, \varepsilon}$ contains none of the points at which
the integrand (\ref{2}) nor the monodromy terms are singular, for $\Re s<0.$ By {\sc Cauchy's} Theorem then, 
\[
\frac{(s-2)\Gamma(1-s)}{2 \pi i}\int_{C_{R,N}^{\delta, \varepsilon}}\frac{\Li_{2}(1-e^{-x})}{x^{2}}
(-x)^{s}\frac{dx}{x} =0.
\]
We proceed to compute this integral, under the assumption that $\Re s<-2.$

Firstly, we know that in the limit as $\varepsilon$ and $\delta$ approach 0 and $R$ nears $\infty$, the integral along $C_{0}$ tends to $-\zeta(s).$ Next we have 
\[
\frac{(s-2)\Gamma(1-s)}{2 \pi i}\int_{R+i \varepsilon}^{R+2 \pi i - i \varepsilon}\frac{\Li_{2}(1-e^{-x})}{x^{2}}
(-x)^{s}\frac{dx}{x},
\]
but passage along this line segment produces a monodromy term from the dilogarithm, of
\begin{equation}
\label{monod}
\frac{2 \pi i\log(1-e^{-x})}{x^{3}}
(-x)^{s}
\end{equation}
which must be taken into account along all subsequent paths.

Now the dilogarithm integral along the translate $C_{0}+2\pi i$ approaches 0 as $\varepsilon$ tends to 0 by {\sc Cauchy's} Theorem. In particular, notice that $(-x)^{s}$ has the same value on the straight line segments of this part of the path, so that the integral of the monodromy term involving $\log (1-e^{-x})$ also vanishes along these straight line segments. What remains to consider then from this portion of $C_{R,N}^{\delta, \varepsilon}$ is
\[
\frac{(s-2)\Gamma(1-s)}{2\pi i}
\int_{\gamma_{1}}
\frac{2 \pi i\log(1-e^{-x})}{x^{2}}
(-x)^{s}\frac{dx}{x}
\]
where $\gamma_{k}$ will denote the (negatively oriented) loop about $2 \pi i k$ for integer $k$, along with 
\[
(s-2)\Gamma(1-s)(2 \pi i)
\int_{2\pi i +\delta + i \varepsilon}^{R+2\pi i + i \varepsilon}
(-x)^{s-2}\frac{dx}{x}, 
\]
the term arising from the monodromy of $\log(1-e^{-x})$ about $2 \pi i.$ As before, this last integrand must be considered along all subsequent subpaths of $C_{R,N}^{\delta, \varepsilon}$. But notice that along the remaining translates of $C_{0}$, this monodromy term (from passage around $2 \pi i$) is 0 (again in the limit as $\varepsilon \rightarrow 0$) by {\sc Cauchy's} Theorem.

Now let $D_{R, n}$ denote the rectangular path 
$R+ 2 n\pi i \rightarrow R+(2N+1)\pi i \rightarrow -R+(2N+1) \pi i \rightarrow -R - (2N+1)\pi i \rightarrow R-(2N+1) \pi i \rightarrow R- 2 n\pi i$ for non-negative integer $n \leq N.$ 
Continuing along $C_{R,N}^{\delta, \varepsilon}$ we find thus that the integrals which are yet to be computed add to
\begin{align*}
&
\frac{(s-2)\Gamma(1-s)}{2 \pi i}\int_{D_{R,0}}\frac{\Li_{2}(1-e^{-x})}{x^{2}}
(-x)^{s}\frac{dx}{x}
\\
&+
\sum_{n=1}^{N}
\frac{(s-2)\Gamma(1-s)}{2\pi i}
\int_{\gamma_{n}}
\frac{n2 \pi i\log(1-e^{-x})}{x^{2}}
(-x)^{s}\frac{dx}{x}
\\
&+\frac{(s-2)\Gamma(1-s)}{2 \pi i}
\sum_{n=1}^{N}
\left\{
\int_{2n\pi i +\delta + i \varepsilon}^{R+2n\pi i + i \varepsilon}\!\!\!\!\!\!\!
n(2 \pi i)^{2}(-x)^{s-2}\frac{dx}{x}\right.
\\
&
+\left.
\int_{D_{R,n}}\sum_{m=1}^{n}m(2 \pi i)^{2}
(-x)^{s-2}\frac{dx}{x}
\right\}
\end{align*}
along with terms with integrand of the form of (\ref{monod}) integrated along the outside contour.

Now along the portion of $C_{R,N}^{\delta, \varepsilon}$ proceeding from the point $-R+(2N+1)\pi i$, further monodromy terms arise from the dilogarithm terms, in this instance the negatives of terms of the form of (\ref{monod}). All such terms themselves exhibit monodromy each time the path traverses a segment of length $2\pi i$ along the line $-R +2N\pi i \rightarrow -R -2N\pi i$ since images of such segments trace out circles of radius $e^{R}$ about $z=1$ under the mapping $x \mapsto 1-e^{-x}=z$.  

These new dilogarithm monodromy terms cancel out those from the first vertical portion of the path, so that all such terms add to zero by when the point $-R$ is reached along $C_{R,N}^{\delta, \varepsilon}.$ 
Below the real axis, negative terms accumulate so that once one reaches the point $-R-(2N+1)\pi i$, the remaining dilogarithm monodromy terms add to
\[
-2 \pi i N \log (1-e^{-x}).
\]
On the other hand, since the number of logarithmic terms as one moves along $-R+i \alpha$ (for real decreasing $\alpha$) decreases from $N$ to $N-1$ to $N-2$ and so on, the sum of the logarithmic monodromy terms number successively $N(N+1)/2; N(N+1)/2+N;$ then $N(N+1)/2 +N+(N-1)$ and so on, until at the point $-R,$ there are $N(N+1)$ such terms. Thereafter, the increasing number of negative logarithmic terms decrease the total number of these second monodromy terms. Eventually, at $-R-(2N+1)\pi i,$ the end of the vertical line, the terms which remain sum to
\[
(2 \pi i)^{2}[1+2+ \ldots +N].
\]
By the same argument as before, the integral coming from the terms $(2 \pi i)^{2}[1+2+ \ldots + N-1]$ is zero around $C_{0}-2N\pi i$, but because of the monodromy of the log term about $-2N\pi i,$ the integral
\[
\frac{(s-2)\Gamma(1-s)}{2\pi i}
\int_{R-2N\pi i - i \varepsilon}^{-2N\pi i +\delta - i \varepsilon}
N(2 \pi i)^{2}(-x)^{s-2}\frac{dx}{x}
\]
{\em{does}} need to be taken into account. Continuing back to the starting point of $C_{R,N}^{\delta, \varepsilon},$ similar terms add to
\[
\frac{(s-2)\Gamma(1-s)}{2 \pi i}
\sum_{n=1}^{N}
\int_{R-2n\pi i - i \varepsilon}^{-2n\pi i +\delta - i \varepsilon}
n(2 \pi i)^{2}(-x)^{s-2}\frac{dx}{x}.
\]

This expression, along with its counterpart from the part of $C_{R,N}^{\delta, \varepsilon}$ with positive imaginary part, is readily computed in the limit as $\delta$ and $\varepsilon$ approach 0 while $R$ tends to $\infty:$ Indeed, by using {\sc Cauchy's} Theorem applied to rectangular contours respectively below and above the real axis, we find that 
\begin{equation}
\label{Z(2)neg}
\lim_{R \rightarrow \infty; \delta, \varepsilon \rightarrow 0}\int_{R-2n\pi i - i \varepsilon}^{-2n\pi i +\delta - i \varepsilon}
n(2 \pi i)^{2}(-x)^{s-2}\frac{dx}{x} = \frac{(2 \pi )^{s}(i)^{s}n^{s-1}}{s-2}
\end{equation}
whereas
\begin{equation}
\label{Z(2)pos}
\lim_{R \rightarrow \infty; \delta, \varepsilon \rightarrow 0}
\int_{2n\pi i +\delta + i \varepsilon}^{R+2n\pi i + i \varepsilon}
n(2 \pi i)^{2}(-x)^{s-2}\frac{dx}{x}
=
-\frac{(2 \pi )^{s}(-i)^{s}n^{s-1}}{s-2}.
\end{equation}
More precisely, taking $\varepsilon \delta \rightarrow 0$, consider the contour $E_{R,n, \kappa,-}$ formed by
\[
R -2 \pi i n \rightarrow -2\pi i n
\rightarrow
-i\kappa
\rightarrow
R-i \kappa
\rightarrow 
R-2 \pi in
\]
where $\kappa>0$ is small. One computes
\[
\int_{-2 \pi in}^{-i \kappa}\frac{(-x)^{s}}{x^{2}}\frac{dx}{x}=
e^{i\frac{\pi}{2}s}\frac{(2 \pi n)^{s-2}}{s-2}
-e^{i \frac{\pi}{2}s}\frac{\kappa^{s-2}}{s-2}
\]
and
\[
\int_{-i\kappa}^{R-i \kappa}\frac{(-x)^{s}}{x^{2}}\frac{dx}{x}=
e^{i \pi s}\frac{(R-i \kappa)^{s-2}}{s-2}
+
e^{i \frac{\pi}{2}s}\frac{\kappa^{s-2}}{s-2}.
\]
while
\[
\left|\int_{R-i \kappa}^{R-2 \pi in}\frac{(-x)^{s}}{x^{2}}\frac{dx}{x}\right|
\leq
K_{E}2 \pi \cdot \max\left\{\left|\frac{(-x)^{s}}{x^{2}}\right|\right\}
\]
where $K_{E}>0$ is constant and the maximum is taken over the straight line segment $R-i \kappa \rightarrow R-2 \pi in$, so that this last integral approaches zero as $R$ grows without bound. 
Adding these integrals in the limit as $R \rightarrow \infty$, we obtain 
\[
(-i^{-2})i^{s}\frac{(2 \pi n)^{s-2}}{s-2}
\]
and by {\sc Cauchy's} Theorem, the integral in (\ref{Z(2)neg}) is the negative of this quantity multiplied by $n(2 \pi i)^{2},$ from which we deduce the equality in (\ref{Z(2)neg}).
A similar computation suffices to show the validity of (\ref{Z(2)pos}).

Adding all such terms of the integral along $C_{R,N}^{\delta, \varepsilon}$ then gives
\begin{eqnarray*}
&&
\frac{(s-2)\Gamma(1-s)}{2\pi i}
\frac{(2 \pi)^{s}[i^{s}-(-i)^{s}]}{s-2}
\sum_{n=1}^{\infty}n^{s-1}
\\
&=&
\Gamma(1-s) 2^{s}\pi^{s-1}\frac{e^{i \frac{\pi}{2}s} - e^{-i\frac{\pi}{2}s}}{2i}\zeta(1-s)
\\
&=&
\Gamma(1-s)2^{s}\pi^{s-1}\sin \left(\frac{\pi s}{2} \right)
\zeta(1-s).
\end{eqnarray*}

Since this term added to $-\zeta(s)$ gives 0, the functional equation follows. 

It remains to be shown that in the limit, the remaining terms approach 0. It is convenient to assume that $2(2N+1) \pi<R.$

We begin with the integral 
\begin{equation}
\label{Liest}
\int_{D_{R,0}}\frac{\Li_{2}(1-e^{x})}{x^{2}}(-x)^{s}\frac{dx}{x}
\end{equation}
where the monodromy terms are ignored. This is most readily approximated by considering separately the two portions of the path $D_{R,0}$ on respective sides of the imaginary axis, say $D_{R,+}$ for the part with non-negative real part and its counterpart $D_{R,-}$ to the left of the imaginary axis.

Now both $D_{R+}$ and $D_{R,-}$ have length
$2R+2(2N+1)\pi <3R$ by the assumption on $N$.

Along $D_{R,+}$, $|\Li_{2}(1-e^{-x})|\leq (\zeta(2)+\Li_{2}(2)\varepsilon_{R})$ where $\varepsilon >0$ approaches 0 as $R \rightarrow \infty$ since $1-e^{-x}$ is close to 0 along $R+i \alpha$ for real $\alpha$, while $\Li_{2}(1-e^{-x})=\Li_{2}(1+e^{-l})$ for positive $l$ along both $R+(2N+1)\pi i \rightarrow (2N+1)\pi i$ and $-(2N+1)\pi i \rightarrow R-(2N+1)\pi i.$ Because the points of $D_{R,+}$ are outside of a circle of radius $R$, also
\[
\left|
\frac{(-x)^{s}}{x^{3}}
\right| 
=
\left|
x
\right|^{\Re s -3} 
\leq R^{\Re s -3}.
\]
Hence, for $C$ and $C'$ denoting positive constants, a bound on the absolute value of the part of (\ref{Liest}) along $D_{R,+}$ is
\[
CR^{\Re s -2}(\zeta(2)+\Li_{2}(2)+\varepsilon_{R})(2R+2(2N+1)\pi)
<
C'R^{\Re s -1} \rightarrow 0
\]
as $R$ tends to $\infty,$ because $\Re s < -2.$

Now on $D_{R,-}$ we again use {\sc Euler} inversion to rewrite the integral as
\begin{equation}
\label{EI2}
\int_{D_{R,-}} \left[ \zeta(2) - \Li_{2}(e^{-x}) + x \log(1-e^{-x}) \right]
\frac{(-x)^{s}}{x^{2}}\frac{dx}{x}.
\end{equation}
Now $\Li_{2}(e^{-x})+\zeta(2)$ may be bounded along $D_{R,-}$ in similar vein to the bound obtained for $\Li_{2}(1-e^{-x})$ along $D_{R,+}$ and it hence follows that the contribution from these first two terms approaches zero as $R \rightarrow 0.$ As for the logarithmic term, notice that for $x \in \lvR,$
\[
\lim_{x\rightarrow -\infty}\frac{\log(1-e^{-x})}{x^{2}} = 0
\]
by {\sc L'H{\^o}pital's} Rule. Consequently, for any given $\eta >0,$ one may choose $R$ sufficiently large so that also
\begin{equation}
\label{logest}
\left|\frac{\log(1-e^{-x})}{x^{2}}\right| < \eta
\end{equation}
for any $x$ on $D_{R,-}.$ Thence, a bound on the third term of (\ref{EI2}) is 
$$
K\eta R^{\Re s+1}(2R+2(2N+1)\pi)<K'\eta R^{\Re s+2}
$$ 
(with positive constants $K$ and $K'$), which tends to 0 as 
$R$ goes to infinity, since we assume that $\Re s<-2.$ 

These estimates also serve to deal with the logarithmic terms coming from the monodromy of the dilogarithm, integrated along the boundary. In particular, it is easily seen that the bound (\ref{logest}) holds along all of $D_{R,0}$ if $\eta$ is chosen suitably. Now the absolute value of the number of such logarithmic terms is at most $N$, so for some positive constants $K''$ and $K'''$, the sum of such integrals is bounded above by
\begin{align*}
N \left|\int_{D_{R,0}} \frac{\log(1-e^{-x})}{x^{2}}(-x)^{s}\frac{dx}{x}\right|
&< 
K''N \eta R^{\Re s}(4R+4(2N+1)\pi)
\\
&<
K'''\eta R^{\Re s +2}
\end{align*}
which as before can be made arbitrarily small.

Next consider the integrals 
\[
\sum_{n=1}^{N}
\int_{\gamma_{n}}
\frac{n2 \pi i\log(1-e^{-x})}{x^{2}}
(-x)^{s}\frac{dx}{x}.
\]
For each $n$, make the change of variables $y=x-2 \pi i n$ to obtain
\begin{align*}
&
\sum_{n=1}^{N}
\int_{\gamma_{0}}
\frac{n2 \pi i\log(1-e^{-y})}{(y+2 \pi i n)^{3}}
(-y-2 \pi i n)^{s}{dy}
\\
&=
\sum_{n=1}^{N}
\int_{\gamma_{0}}
\frac{n2 \pi iy\log(1-e^{-y})}{(y+2 \pi i n)^{3}}
(-y-2 \pi i n)^{s}\frac{dy}{y}.
\end{align*}
Along the real axis, {\sc L'H{\^o}pital's} Rule gives
\[
\lim_{y \rightarrow 0}
y \log(1-e^{-y}) = 0,
\]
so that any point $z$ along $\gamma_{0}$ has 
\[
|z \log(1-e^{-z})| \leq |y_{0}(\log(1-e^{-y_{0}})+2 \pi i)|< \eta'
\]
where $y_{0}$ lies on $\gamma_{0} \cap \lvR_{>0}$ and $\eta'>0$ approaches 0 with  $\delta$. This shows that
\begin{eqnarray*}
&&
\left|\int_{\gamma_{0}}y \log (1-e^{-y}) \frac{(-y-2 \pi in)^{s}}{(y+2 \pi i n)^{3}} \frac{dy}{y}\right|\\
&\leq&
\eta'\left|
\int_{\gamma_{0}}\frac{(-y-2 \pi in)^{s}}{(y+2 \pi i n)^{3}} \frac{dy}{y}\right|\\
&=&
\eta'
(2 \pi n)^{\Re s-3}
\end{eqnarray*}
using the calculus of residues for the last computation. Adding now all such terms along with the similar integrals along loops around $-2 \pi i n$ for $n \geq 0,$ the bound computes to
$\eta' 2(2\pi)^{\Re s -3}\zeta(2-s)$ and this evidently tends to 0 as $\eta'$ does.

Finally we dispense with the terms arising from monodromy of the logarithmic terms, integrated along the outside of the contour. The largest number of such terms at any point along the contour is $N(N+1)$ so that the sum of all such integrals is certainly bounded by
\[
N(N+1)(4R+4(2N+1)\pi)R^{\Re s -2}
<
R^{\Re s +1}
\]
which becomes arbitrarily small as $R$ grows without bound.

\bibliographystyle{alpha}	
\bibliography{bibliog}	

{Sheldon T Joyner, Mathematics Department, Brandeis University, 
415 South St, Waltham, MA 02453}
\\
{\em  email:} {\tt{joyner@brandeis.edu}}

\end{document}